\documentclass[12pt]{article}

\usepackage[left=2cm,right=2cm,top=3cm,bottom=3cm]{geometry}

\usepackage{amsmath, amssymb, latexsym, amsthm}
\usepackage{fullpage}

\usepackage{tikz}
\usetikzlibrary{shapes}
\usetikzlibrary{calc,decorations.pathreplacing}

\newtheorem{theorem}{Theorem}[section] 
\newtheorem{theorem*}{Theorem} 

\newtheorem{corollary}[theorem]{Corollary}
\newtheorem{lemma}[theorem]{Lemma}
\newtheorem{conjecture}[theorem]{Conjecture}

\theoremstyle{definition}

\newtheorem{observation}[theorem]{Observation}

\theoremstyle{remark}

{\end{enumerate}}

\newcommand{\CL}[1]{\left\lceil #1 \right\rceil}

\newcommand{\FL}[1]{\left\lfloor #1 \right\rfloor}

\expandafter\ifx\csname dplus\endcsname\relax \csname newbox\endcsname\dplus\fi
\expandafter\ifx\csname dplustemp\endcsname\relax
\csname newdimen\endcsname\dplustemp\fi
\setbox\dplus=\vtop{\vskip -8pt\hbox{%
    \kern .2em
    \special{pn 6}%
    \special{pa 0 50}%
    \special{pa 50 100}%
    \special{fp}%
    \special{pa 50 100}%
    \special{pa 100 50}%
    \special{fp}%
    \special{pa 100 50}%
    \special{pa 50 0}%
    \special{fp}%
    \special{pa 50 100}%
    \special{pa 50 0}%
    \special{fp}%
    \special{pa 50 0}%
    \special{pa 0 50}%
    \special{fp}%
    \special{pa 0 50}%
    \special{pa 100 50}%
    \special{fp}%
    \hbox{\vrule depth0.080in width0pt height 0pt}%
    \kern .7em
  }%
}%
\def\gjoin{\copy\dplus}

\def\esub{\subseteq}
\def\nul{\varnothing}
\def\Gb{\overline{G}}
\def\Kb{\overline{K}}
\def\VEC#1#2#3{#1_{#2},\ldots,#1_{#3}}
\def\th{\hat t}

\title{Uniquely cycle-saturated graphs}
\author{Paul S.\ Wenger\footnotemark[1]\ \ and Douglas B.\ West\footnotemark[2]}

\begin{document}

\maketitle

\begin{abstract}
Given a graph $F$, a graph $G$ is {\it uniquely $F$-saturated} if $F$ is not a
subgraph of $G$ and adding any edge of the complement to $G$ completes
exactly one copy of $F$.  In this paper we study uniquely $C_t$-saturated
graphs.  We prove the following:
(1) a graph is uniquely $C_5$-saturated if and only if it is a friendship
graph.
(2) There are no uniquely $C_6$-saturated graphs or uniquely $C_7$-saturated
graphs.
(3) For $t\ge6$, there are only finitely many uniquely $C_t$-saturated graphs
(we conjecture that in fact there are none).

\end{abstract}

\renewcommand{\thefootnote}{\fnsymbol{footnote}}
\footnotetext[1]{
School of Mathematical Sciences, Rochester Institute of Technology, Rochester, NY;
{\tt pswsma@rit.edu}.}
\footnotetext[2]{
Departments of Mathematics, Zhejiang Normal University, China, and University of Illinois, USA ,
{\tt west@math.uiuc.edu}. Research supported by Recruitment Program of Foreign Experts, 1000 Talent Plan,
State Administration of Foreign Experts Affairs, China}
\renewcommand{\thefootnote}{\arabic{footnote}}

\baselineskip18pt

\section{Introduction}
Given a graph $F$, a graph $G$ is {\it $F$-saturated} if $F$ is not a subgraph
of $G$ but is a subgraph of $G+e$ for every edge $e$ in the complement $\Gb$ of
$G$.  In 1907, Mantel~\cite{Mantel} proved that the $n$-vertex $K_3$-saturated
graph with the most edges is $K_{\CL{n/2},\FL{n/2}}$.  Tur\' an~\cite{Turan}
generalized this result, proving that the $n$-vertex $K_t$-saturated graph with
the most edges is the complete $(t-1)$-partite graph with partite sets as
balanced as possible.  Erd\H os, Hajnal, and Moon~\cite{EHM} proved that the
$n$-vertex $K_t$-saturated graph with the fewest edges is
$K_{t-2}\gjoin \Kb_{n-t+2}$, where the {\it join} $G\gjoin H$ of graphs $G$ and
$H$ consists of the disjoint union of $G$ and $H$ plus edges connecting all
vertices of $G$ to all vertices of $H$.

There is an important distinction between $K_t$-saturated graphs with the most
and the fewest edges.  When an edge is added to a largest $n$-vertex
$K_t$-saturated graph, roughly $\left(\frac{n}{t-1}\right)^{t-2}$ copies of
$K_t$ are formed.  In contrast, when an edge is added to a smallest $n$-vertex
$K_t$-saturated graph, exactly one copy of $K_t$ is formed.  Given a graph $F$
and an $F$-saturated graph $G$, we say that $G$ is {\it uniquely $F$-saturated}
if the addition of any edge to $G$ completes exactly one copy of $F$.

Questions about uniquely $F$-saturated graphs focus on their existence.
Cooper, LeSaulnier, Lenz, Wenger, and West~\cite{CLLWW}  initiated the study
of uniquely $F$-saturated graphs by determining all uniquely $C_4$-saturated
graphs, where $C_t$ denotes the $t$-vertex cycle; there are exactly 10 such
graphs.  They also observed that a graph is uniquely $C_3$-saturated if and
only if it is a star or a Moore graph of diameter $2$.

Stars have a dominating vertex, but Moore graphs of diameter $2$ do not.
If $G$ is uniquely $K_{t}$-saturated, then $K_m\gjoin\,G$ is uniquely
$K_{m+t}$-saturated and has dominating vertices.  Cooper~\cite{Cooper}
conjectured that for $t\ge2$, only finitely many $K_t$-saturated graphs
have no dominating vertices.  Hartke and Stolee~\cite{HStol} computationally found
new examples for small $t$ of $K_t$-saturated graphs without dominating vertices
and found two constructions of $K_t$-saturated graphs without dominating vertices
based on Cayley graphs, each valid for infinitely many $t$.

Berman, Chappell, Faudree, Gimble, and Hartman~\cite{BCFGH} studied uniquely
tree-saturated graphs.  They proved if $T$ is a tree, then there exist
infinitely many uniquely $T$-saturated graphs if and only if $T$ is a balanced
double star.

When $F$ has $t$ vertices, every complete graph with fewer than $t$ vertices
trivially is uniquely $F$-saturated, since there are no edges to consider
adding.  Let a uniquely $C_t$-saturated graph be {\it nontrivial} if it has
at least $t$ vertices.
In Section~\ref{Sec:Lemmas} we establish structural lemmas about such graphs.
In Section~\ref{Sec:C_5} we prove that the nontrivial uniquely $C_5$-saturated
graphs are precisely the graphs consisting of edge-disjoint triangles with
one common vertex (adjacent to all others).  Such graphs are also called
{\it friendship graphs}, because they are the graphs in which every two
vertices have exactly one common neighbor (proved initially by Erd\H{o}s,
R\'enyi, and S\'os~\cite{ERS} and later reproved by others).  In
Section~\ref{Sec:C_6C_7} we prove that there are no nontrivial uniquely
$C_6$-saturated graphs or uniquely $C_7$-saturated graphs.  Finally, in
Section~\ref{Sec:finite} we prove the following theorem.
\begin{theorem}\label{thm:finite}
For $t\ge 6$, there are finitely many uniquely $C_t$-saturated graphs.
\end{theorem}


In light of our results, we make the following conjecture.

\begin{conjecture}~\label{conj:none}
For $t\ge 6$ there are no nontrivial uniquely $C_t$-saturated graphs.
\end{conjecture}

We have verified Conjecture~\ref{conj:none} for $t=8$, but the proof is quite
long and does not contain any new ideas beyond those used in the proofs of
Theorems~\ref{thm:C6} and~\ref{thm:C7}; thus we do not include the proof here.

\section{Structural lemmas}\label{Sec:Lemmas}

In keeping with the convention of using $k$-cycle for a copy of $C_k$, we refer
to a path with $k$ vertices as a {\it $k$-path}.
We use $\langle v_1,\ldots,v_k\rangle$ to denote the $k$-path with vertices
$v_1,\ldots, v_k$ indexed in order.
We use $[v_1,\ldots,v_{k}]$ to denote the $k$-cycle with vertices
$v_1,\ldots, v_{k}$ indexed in order.
For vertices $x$ and $y$ in a graph $G$, we use $d_G(x,y)$ to denote the
distance between $x$ and $y$ and $d_G(x)$ for the degree of $x$.

We begin with an elementary observation about uniquely $C_t$-saturated graphs.

\begin{observation}
Any two vertices in a uniquely $C_t$-saturated graph are the endpoints of at
most one $t$-path, and such a path exists if and only if they are not adjacent.
\end{observation}

A {\it block} in a graph is a maximal subgraph not having a cut-vertex.  Thus
it is a maximal $2$-connected subgraph or has a single edge that is a cut-edge.

\begin{lemma}\label{smallblock}
Every block in a uniquely $C_{t}$-saturated graph is uniquely $C_t$-saturated.
In particular, blocks with fewer than $t$ vertices are complete graphs.
\end{lemma}

\begin{proof}
Let $x$ and $y$ be nonadjacent vertices in a non-complete block $B$ of such a
graph $G$.  Since $B$ is a maximal $2$-connected subgraph, the unique $t$-path
in $G$ with endpoints $x$ and $y$ is contained in $B$.  Thus $|V(B)|\ge t$,
and $B$ is uniquely $C_t$-saturated.
\end{proof}

We next bound the size of complete blocks in nontrivial uniquely
$C_{t}$-saturated graphs.

\begin{lemma}\label{completeblock}
Every complete block in a nontrivial uniquely $C_{t}$-saturated graph has at
most three vertices.
\end{lemma}

\begin{proof}
The claim is trivial for $t=3$, so assume $t\ge 4$.  Let $G$ be a non-complete
graph.  Let $B$ and $B'$ be blocks in $G$, with a common vertex $v$, such that
$B$ is a complete graph and has at least four vertices.  Let $u$ and $x$ be
vertices other than $v$ in $B$ and $B'$, respectively.  The vertices $u$ and
$x$ are nonadjacent, and every $t$-path $P$ with endpoints $u$ and $x$ contains
$v$.

If $P$ is unique, then the portion of $P$ in $B$ must have length $1$, since
$B$ has at least four vertices.  Since $t\ge4$, this implies that $P$ has at
least two edges in $B'$.  Hence the neighbor $x'$ of $x$ on $P$ is not $v$.
Now $x'$ and $u$ are not adjacent, but there are multiple $3$-paths in $B$
with endpoints $u$ and $v$, so $G$ is not uniquely $C_{t}$-saturated.
\end{proof}

Using Lemma~\ref{completeblock}, we can restrict our attention to $2$-connected
graphs when $t\ge 6$.

\begin{lemma}\label{twoblock}
If $t \geq 6$, then every nontrivial uniquely $C_{t}$-saturated graph contains
a block that is not a complete graph.  In fact, no two blocks with a common
vertex are complete.
\end{lemma}

\begin{proof}
Let $G$ be a nontrivial uniquely $C_t$-saturated graph.  By
Lemma~\ref{completeblock}, the union of two complete blocks with a common
vertex $v$ has at most five vertices.  Hence if $u$ and $x$ are vertices 
of these blocks other than $v$, then $u$ and $x$ are nonadjacent but are
not the endpoints of a $t$-path.  The contradiction implies that no two
neighboring blocks can be complete.
\end{proof}

For $t\ge6$, Lemma~\ref{twoblock} reduces Conjecture~\ref{conj:none} to
the consideration of $2$-connected graphs.

\begin{corollary}\label{2conn}
For $t\ge 6$, if there are no $2$-connected nontrivial uniquely $C_t$-saturated
graphs, then there are no nontrivial uniquely $C_t$-saturated graphs.
\end{corollary}

We next forbid certain subgraphs, aiming to forbid certain cycle lengths.
Let $H_{m,\ell}$ be the graph that consists of a $2m$-cycle with a pendant path
of length $\ell$ (see Figure~\ref{fig:Hkl}).

\begin{figure}[h]
\centering
\begin{tikzpicture}

\fill (0,.5) circle (4pt);
\fill (.5,1) circle (4pt);
\fill (.5,0) circle (4pt);
\fill (1,.5) circle (4pt);
\fill (1.8,.5) circle (4pt);
\fill (2.6,.5) circle (4pt);
\fill (3.4,.5) circle (4pt);
\fill (4.2,.5) circle (4pt);
\draw[line width=2pt] (0,.5)--(.5,1);
\draw[line width=2pt] (0,.5)--(.5,0);
\draw[line width=2pt] (.5,1)--(1,.5);
\draw[line width=2pt] (.5,0)--(1,.5);
\draw[line width=2pt] (1,.5)--(4.2,.5);
\node at (1.5,-.6) {$H_{2,4}$};
\fill (5.7,.5) circle (4pt);
\fill (6.2,1) circle (4pt);
\fill (6.2,0) circle (4pt);
\fill (7,1) circle (4pt);
\fill (7,0) circle (4pt);
\fill (7.5,.5) circle (4pt);
\fill (8.3,.5) circle (4pt);
\fill (9.1,.5) circle (4pt);
\fill (9.9,.5) circle (4pt);
\draw[line width=2pt] (5.7,.5)--(6.2,1);
\draw[line width=2pt] (5.7,.5)--(6.2,0);
\draw[line width=2pt] (6.2,1)--(7,1);
\draw[line width=2pt] (6.2,0)--(7,0);
\draw[line width=2pt] (7,1)--(7.5,.5);
\draw[line width=2pt] (7,0)--(7.5,.5);
\draw[line width=2pt] (7.5,.5)--(9.9,.5);
\node at (7.6,-.6) {$H_{3,3}$};
\fill (11.4,.5) circle (4pt);
\fill (11.9,1) circle (4pt);
\fill (11.9,0) circle (4pt);
\fill (12.7,1) circle (4pt);
\fill (12.7,0) circle (4pt);
\fill (13.5,1) circle (4pt);
\fill (13.5,0) circle (4pt);
\fill (14,.5) circle (4pt);
\fill (14.8,.5) circle (4pt);
\fill (15.6,.5) circle (4pt);
\draw[line width=2pt] (11.4,.5)--(11.9,1);
\draw[line width=2pt] (11.4,.5)--(11.9,0);
\draw[line width=2pt] (11.9,1)--(13.5,1);
\draw[line width=2pt] (11.9,0)--(13.5,0);
\draw[line width=2pt] (13.5,1)--(14,.5);
\draw[line width=2pt] (13.5,0)--(14,.5);
\draw[line width=2pt] (14,.5)--(15.6,.5);
\node at (13.5,-.6) {$H_{4,2}$};
\end{tikzpicture}
\caption{\label{fig:Hkl} Forbidden subgraphs for uniquely $C_7$-saturated graphs.}
\end{figure}
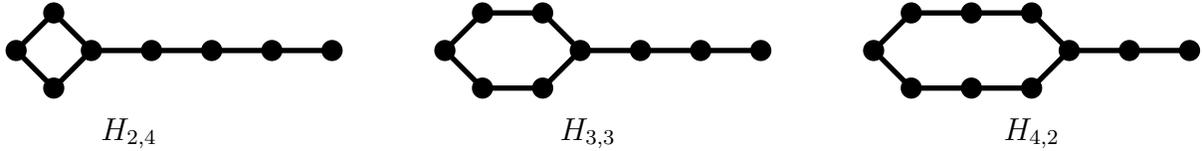

\begin{lemma}\label{cycpath}
If $k<t$ with $t\ge3$, then no uniquely $C_{t}$-saturated graph contains
$H_{k,t-k-1}$.
\end{lemma}

\begin{proof}
The diameter of $H_{k,t-k-1}$ is $t-1$, and there are two $t$-paths connecting
two vertices at distance $t-1$.
\end{proof}

\begin{lemma}\label{lem:C2t-2C2t-4}
For $t\ge 5$, a uniquely $C_t$-saturated graph $G$ cannot contain $C_{2t-2}$ or
$C_{2t-4}$.
\end{lemma}

\begin{proof}
Note that $C_{2t-2}=H_{t-1,0}$, so Lemma~\ref{cycpath} applies.  If $G$
contains $C_{2t-4}$, then avoiding $H_{t-2,1}$ requires $|V(G)|=2t-4$.  Let
$C$ be a spanning cycle in $G$, with $C=[\VEC v0{2t-5}]$ (indices taken modulo
$2t-4$).  If $G$ contains a chord of $G$, then it creates cycles of lengths
$l$ and $2t-2-l$, for some $l$.  If $l=2k$, then $H_{k,t-k-1}\esub G$.

If $l=2k+1$ is odd, then we may assume by symmetry that the chord is
$v_kv_{-k}$.  Now $G$ contains two $t$-paths with endpoints $v_0$ and 
$v_{t-2}$, using the chord in opposite directions.

Hence $G=C_{2t-4}$, but now opposite vertices are not connected by any $t$-path.
\end{proof}

The {\it girth} of a graph is the minimum length of a cycle in it.

\begin{lemma}\label{girth}
For $t\ge 5$, a uniquely $C_{t}$-saturated graph $G$ has girth at most $t+1$.
\end{lemma}

\begin{proof}
Let $x$ and $y$ be two vertices in $G$ such that $d_G(x,y)=2$, and let $z$ be a
common neighbor of $x$ and $y$.  Let $P$ be the unique $t$-path with endpoints
$x$ and $y$.  If $P$ does not contain $z$, then the union of $P$ and the path
$\langle x,z,y\rangle$ is a $(t+1)$-cycle.  If $P$ contains $z$, then the union
of $P$ and $\langle x,z,y\rangle$ contains a cycle with length at most $t$.
\end{proof}

Two vertices having the same neighborhood are {\it twins}.

\begin{lemma}\label{twins}
For $t\ge4$, a uniquely $C_t$-saturated graph cannot contain twins.
\end{lemma}

\begin{proof}
Let $x$ and $y$ be twins in a graph $G$; note that twins are nonadjacent.
If $d(x)=d(y)=1$, then the only cycle completed by added $xy$ is a $3$-cycle,
so $G$ is not uniquely $C_t$-saturated.  If $d(x)=d(y)\ge 2$ and there is a
$t$-path $P$ with endpoints $x$ and $y$, then let $x'$ and $y'$ be the
neighbors of $x$ and $y$ on $P$, respectively.  Because $x$ and $y$ are twins,
$x'y,xy'\in E(G)$.  Reversing the central $(t-2)$-path of $P$ yields a second
$t$-path with endpoints $x$ and $y$ containing the edges $xy'$ and $x'y$.
Thus $G$ is not uniquely $C_t$-saturated.
\end{proof}

A {\it chordal path} of a cycle $C$ is a path of length at least $2$ whose
endpoints are in $C$ and whose internal vertices are not in $C$.

\begin{lemma}\label{lem:noevencyc}
For $t\ge 6$, every nontrivial uniquely $C_t$-saturated graph $G$ contains
an even cycle of length at most $2t-6$.
\end{lemma}

\begin{proof}
By Lemma~\ref{girth}, $G$ has girth at most $t+1$.  By
Lemma~\ref{lem:C2t-2C2t-4}, $G$ does not contain $C_{2t-2}$ or $C_{2t-4}$.
Since $t+2\le 2t-2$, we may assume that the girth of $G$ is odd and that $G$
has no even cycle of length at most $2t-2$.

First suppose that $G$ contains a cycle $C$ of length $2k+1$ such that
$2\le k\le \FL{t/2}$.  The prohibition of short even cycles implies that $C$
has no chord.  Let $x$ and $y$ be nonconsecutive vertices in $C$ such that
$d_C(x,y)\neq t-1$; thus $x$ and $y$ are not adjacent.  A $t$-path with
endpoints $x$ and $y$ contains a chordal path of $C$ with length at most $t-1$.
Combining this chordal path with a path from $x$ to $y$ along $C$ yields an
even cycle with length at most $2t-2$ in $G$, a contradiction.

Now suppose that $G$ contains a $3$-cycle but no cycle of length $2k+1$ with
$2\le k\le\FL{t/2}$.  Let $[x,y,z]$ be a $3$-cycle $C$, and let $x$ be a vertex
of $C$ having a neighbor $x'\notin V(C)$.  Since by assumption $G$ has no
$4$-cycle, $x'y\notin E(G)$.  Hence $G$ has a $t$-path $P$ with endpoints $x'$
and $y$.  Now $P$ contains one of the following: a subpath of length at least
$3$ with endpoints $x'$ and $x$, a chordal path of length at least $2$
connecting two vertices in $V(C)$, or a path connecting $x'$ and a vertex in
$\{y,z\}$ with all internal vertices outside $C$.  In all cases, $G$ contains
an even cycle of length at most $t+2$ or an odd cycle of length $2k+1$ with
$2\le k\le\FL{t/2}$.
\end{proof}

In light of Lemma~\ref{lem:noevencyc} guaranteeing an even cycle of length
at most $2t-6$, an approach to proving Conjecture~\ref{conj:none} that there is
no uniquely $C_t$-saturated graph for $t\ge6$ is to prove that such a graph has
no such even cycle.  Although we cannot completely exclude $(2t-6)$-cycles
for all $t$, we can greatly restrict the graphs that contain them.  We state a
general lemma without proof, because we present the proof of
Conjecture~\ref{conj:none} only through $t=7$.  The ad hoc proof excluding
$8$-cycles when $t=7$ is shorter than the general proof of this lemma.

\begin{lemma}\label{2t-6}
For $t\ge7$, if a nontrivial uniquely $C_t$-saturated graph $G$ contains
a $(2t-6)$-cycle $C$ and $R=V(G)-V(C)$, then (1) $G[R]$ has no edges,
(2) every vertex of $R$ has exactly two neighbors on $C$, separated by odd
distance (at least $3$) along $C$, and (3) all vertices of $R$ have the same
distance along $C$ between their neighbors on $C$.  Also, (4) all chords of $C$
join vertices at even distance along $C$.
\end{lemma}

\section{Uniquely $C_{5}$-saturated graphs}\label{Sec:C_5}

As mentioned in the introduction, the Friendship Theorem of
Erd\H{o}s, R\'enyi, and S\' os~\cite{ERS} states that if every two vertices
in a graph have exactly one common neighbor, then some vertex is adjacent
to all others.  As they noted, this immediately implies that the graph
consists of edge-disjoint triangle with one common vertex.

In such a graph, there is only one type of missing edge, joining two of the
triangles.  Adding this to two edges from each of the two triangles completes a
unique $5$-cycle.  Hence friendship graphs are uniquely $C_5$-saturated. 

\begin{theorem}\label{them:C5}
A graph is a nontrivial uniquely $C_5$-saturated graph if and only if it is a
friendship graph with at least five vertices.
\end{theorem}

\begin{proof}
We have noted that the condition is sufficient.  For the converse, let $G$ be a
nontrivial uniquely $C_5$-saturated graph.  Our proof depends on five graphs
that cannot be subgraphs of $G$.  Already $H_{2,2}$ and $H_{3,1}$ are
excluded by Lemma~\ref{cycpath}.  Let $F$ consist of $K_4$ plus a pendant edge
at one vertex, and let $F'$ consist of the $5$-vertex friendship graph plus a
pendant edge at a vertex of degree $2$.  Figure~\ref{fig:C5forb} illustrates
that $F$, $F'$, and the complete bipartite graph $K_{2,3}$ are all forbidden
as subgraphs of $G$, since each has a nonadjacent vertex pair that when added
completes at least two $5$-cycles.

\begin{figure}[h]
\centering
\begin{tikzpicture}

\fill (5,.5) circle (4pt);
\fill (5.5,1) circle (4pt);
\fill (5.5,0) circle (4pt);
\fill (6,.5) circle (4pt);
\fill (7,.5) circle (4pt);
\draw[line width=2pt] (5,.5)--(5.5,0);
\draw[line width=2pt] (5,.5)--(5.5,1);
\draw[line width=2pt] (6,.5)--(5.5,0);
\draw[line width=2pt] (6,.5)--(5.5,1);
\draw[line width=2pt] (7,.5)--(5.5,0);
\draw[line width=2pt] (7,.5)--(5.5,1);
\draw[dotted, line width=2pt] (5,.5)--(6,.5);
\node at (6,-.6) {$K_{2,3}$};
\fill (9,0) circle (4pt);
\fill (9,1) circle (4pt);
\fill (10,1) circle (4pt);
\fill (10,0) circle (4pt);
\fill (11,0) circle (4pt);
\draw[line width=2pt] (9,0)--(9,1);
\draw[line width=2pt] (9,0)--(10,0);
\draw[line width=2pt] (9,0)--(10,1);
\draw[line width=2pt] (9,1)--(10,0);
\draw[line width=2pt] (9,1)--(10,1);
\draw[line width=2pt] (10,0)--(10,1);
\draw[line width=2pt] (10,0)--(11,0);
\draw[dotted, line width=2pt] (11,0) to [bend right=30] (10,1);
\node at (10,-.6) {$F$};

\fill (13,0) circle (4pt);
\fill (13,1) circle (4pt);
\fill (14,.5) circle (4pt);
\fill (15,0) circle (4pt);
\fill (15,1) circle (4pt);
\fill (16,1) circle (4pt);
\draw[line width=2pt] (13,0)--(13,1);
\draw[line width=2pt] (13,0)--(14,.5);
\draw[line width=2pt] (13,1)--(14,.5);
\draw[line width=2pt] (14,.5)--(15,0);
\draw[line width=2pt] (14,.5)--(15,1);
\draw[line width=2pt] (15,0)--(15,1);
\draw[line width=2pt] (15,1)--(16,1);
\draw[dotted, line width=2pt] (16,1) to [bend right=35] (13,1);
\node at (14.5,-.6) {$F'$};
\end{tikzpicture}
\caption{\label{fig:C5forb}Three graphs forbidden as subgraphs of uniquely
$C_5$-saturated graphs.}
\end{figure}
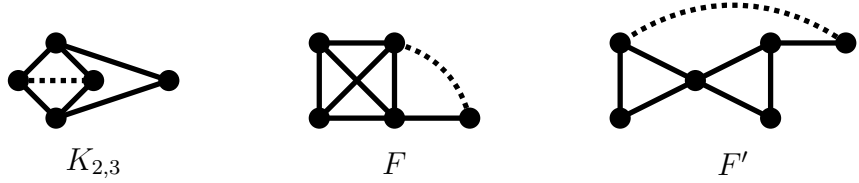

By Lemma~\ref{girth}, $G$ has girth at most $6$; by Lemma~\ref{lem:C2t-2C2t-4},
$G$ has no $6$-cycle.  By definition, $G$ has no $5$-cycle.

Suppose first that $G$ has a $4$-cycle; let $S$ be its vertex set.  Let
$R=V(G)-S$.  Because $G$ is connected and $H_{2,2}\not\esub G$, each vertex in
$R$ has a neighbor in $S$, and $R$ is an independent set.  Because
$C_5,K_{2,3}\not\esub G$, each vertex in $R$ has exactly one neighbor in $S$.
Therefore $S$ is the vertex set of a block in $G$, and by
Lemma~\ref{smallblock} $S$ is a clique.  Since $F\not\esub G$, we conclude that
$R= \emptyset$ and $G=K_{4}$.  Thus no nontrivial uniquely $C_5$-saturated
graph has a $4$-cycle.

We conclude that $G$ has no $4$-cycle but has a $3$-cycle, say $[x,y,z]$.
Since $G$ is nontrivial and connected, we may assume by symmetry that $x$ has a
neighbor $x'$ not in $\{y,z\}$.  Since $G$ has no $4$-cycle, $y$ and $z$ are
not adjacent to $x'$.  Since $G$ has no $4$-cycle or $6$-cycle, the unique
$5$-path $P$ with endpoints $x'$ and $y$ contains $x$.  It must be
$\langle x',w,x,z,y\rangle$, where $w$ is a common neighbor of $x'$ and $x$.
Since $F'\not\esub G$, we conclude that $y$, $z$, $x'$, and $w$ have no
other neighbors in $G$.  Repeating the argument shows that $x$ is a dominating
vertex and $G-x$ is a disjoint union of copies of $K_2$, so $G$ is a friendship
graph with at least five vertices.
\end{proof}

The case in Theorem~\ref{them:C5} where $G$ has no $4$-cycle shows why the
proof of Lemma~\ref{lem:noevencyc} is not valid for $t=5$.  The common neighbor
of $x'$ and $x$ yields the $5$-path with endpoints $x'$ and $y$
without creating a $4$-cycle.

\section{Uniquely $C_{6}$- and $C_7$-saturated graphs}\label{Sec:C_6C_7}

In this section, we prove that there are no nontrivial uniquely
$C_{6}$-saturated or uniquely $C_7$-saturated graphs.  Our proofs depend on
successively forbidding cycles of various lengths.

\begin{theorem}\label{thm:C6}
There are no nontrivial uniquely $C_6$-saturated graphs.
\end{theorem}

\begin{proof}
Let $G$ be a uniquely $C_6$-saturated graph.  By Corollary~\ref{2conn}, we may
assume that $G$ is $2$-connected.
By Lemma~\ref{lem:C2t-2C2t-4}, $G$ does not contain $C_{10}$ or $C_8$.
By Lemma~\ref{lem:noevencyc}, $G$ contains an even cycle of length at most $6$.
By definition, $G$ does not contain $C_6$.  Hence $G$ contains $C_4$.
Let $S$ be the vertex set of a $4$-cycle in $G$, and let $R=V(G)-S$.

First suppose that $G[R]$ contains a 3-path $\langle u,z,v\rangle$.  Since $G$
is $2$-connected, two disjoint paths connect $\{u,z,v\}$ to $S$.  Choosing
shortest such paths, one has $u$ or $v$ as an endpoint, yielding
$H_{2,3}\esub G$.  This contradicts Lemma~\ref{cycpath}; we conclude 
$\Delta(G[R])\le 1$.

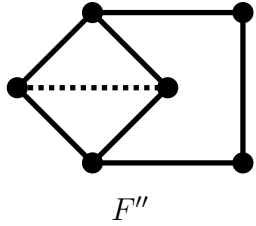
\begin{figure}[htb]
 \begin{center}
 \begin{tikzpicture}
\fill (0,1) circle (4pt);
\fill (1,2) circle (4pt);
\fill (1,0) circle (4pt);
\fill (2,1) circle (4pt);
\fill (3,2) circle (4pt);
\fill (3,0) circle (4pt);
\draw[line width=2pt] (0,1)--(1,2);
\draw[line width=2pt] (0,1)--(1,0);
\draw[line width=2pt] (2,1)--(1,2);
\draw[line width=2pt] (2,1)--(1,0);
\draw[line width=2pt] (1,2)--(3,2);
\draw[line width=2pt] (3,2)--(3,0);
\draw[line width=2pt] (1,0)--(3,0);
\draw[dotted, line width=2pt] (0,1) to  (2,1);
\node at (1.5,-.6) {$F''$};
\end{tikzpicture}
 \end{center}
\vspace{-1pc}
\caption{A forbidden subgraph for uniquely $C_{6}$-saturated graphs.\label{c6forb1}}
\end{figure}

Suppose $uv\in E(G[R])$.  Since $G$ is $2$-connected and the only edges leaving
$\{u,v\}$ go to $S$, there are distinct vertices $x,y\in S$ such that
$\langle x,u,v,y\rangle$ is a $4$-path.  Since $G$ cannot contain $C_6$, it
contains the graph $F''$ in Figure~\ref{c6forb1}.  Since $F''$ has two
nonadjacent vertices connected by more than one $6$-path, $G$ is not uniquely
$C_6$-saturated.

We may therefore assume that $R$ is an independent set.  Since $G$ is
$2$-connected, each vertex in $R$ has at least two neighbors in $S$.
If $|V(G)|\ge 6$, then let $u$ and $v$ be vertices in $R$.  If two neighbors
of each can be chosen in $S$ that are not the same pair, then $G[S\cup\{u,v\}]$
contains a $6$-cycle or two $6$-paths with endpoints $u$ and $v$, as shown in
Figure~\ref{c6forb2}.  Hence $u$ and $v$ have degree $2$ and have the same two
neighbors in $S$.  This makes them twins, which is forbidden by
Lemma~\ref{twins}.  We conclude that $G$ contains at most five vertices, which
yields $G\in \{K_4,K_5\}$.  We conclude that there are no nontrivial uniquely
$C_6$-saturated graphs.
\end{proof}

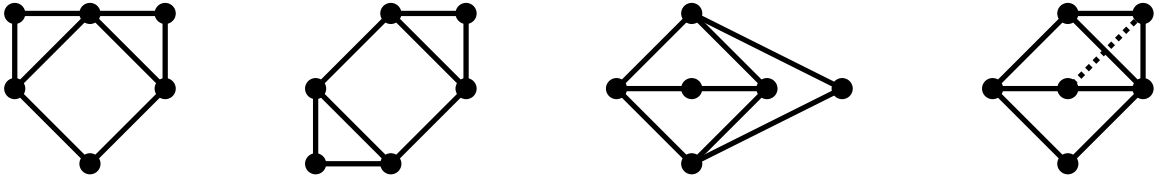
\begin{figure}[htb]
 \begin{center}
 \begin{tikzpicture}
\fill (0,1) circle (4pt);
\fill (1,2) circle (4pt);
\fill (1,0) circle (4pt);
\fill (2,1) circle (4pt);
\fill (0,2) circle (4pt);
\fill (2,2) circle (4pt);
\draw[line width=2pt] (0,1)--(1,2);
\draw[line width=2pt] (0,1)--(1,0);
\draw[line width=2pt] (2,1)--(1,2);
\draw[line width=2pt] (2,1)--(1,0);
\draw[line width=2pt] (0,1)--(0,2);
\draw[line width=2pt] (1,2)--(0,2);
\draw[line width=2pt] (1,2)--(2,2);
\draw[line width=2pt] (2,1)--(2,2);
\fill (8,1) circle (4pt);
\fill (9,2) circle (4pt);
\fill (9,0) circle (4pt);
\fill (10,1) circle (4pt);
\fill (9,1) circle (4pt);
\fill (11,1) circle (4pt);
\draw[line width=2pt] (8,1)--(9,2);
\draw[line width=2pt] (8,1)--(9,0);
\draw[line width=2pt] (10,1)--(9,2);
\draw[line width=2pt] (10,1)--(9,0);
\draw[line width=2pt] (8,1)--(9,1);
\draw[line width=2pt] (9,1)--(10,1);
\draw[line width=2pt] (11,1)--(9,2);
\draw[line width=2pt] (11,1)--(9,0);
\fill (4,1) circle (4pt);
\fill (5,2) circle (4pt);
\fill (5,0) circle (4pt);
\fill (6,1) circle (4pt);
\fill (4,0) circle (4pt);
\fill (6,2) circle (4pt);
\draw[line width=2pt] (4,1)--(5,2);
\draw[line width=2pt] (4,1)--(5,0);
\draw[line width=2pt] (6,1)--(5,2);
\draw[line width=2pt] (6,1)--(5,0);
\draw[line width=2pt] (4,1)--(4,0);
\draw[line width=2pt] (5,0)--(4,0);
\draw[line width=2pt] (5,2)--(6,2);
\draw[line width=2pt] (6,1)--(6,2);
\fill (13,1) circle (4pt);
\fill (14,2) circle (4pt);
\fill (14,0) circle (4pt);
\fill (15,1) circle (4pt);
\fill (14,1) circle (4pt);
\fill (15,2) circle (4pt);
\draw[line width=2pt] (13,1)--(14,2);
\draw[line width=2pt] (13,1)--(14,0);
\draw[line width=2pt] (15,1)--(14,2);
\draw[line width=2pt] (15,1)--(14,0);
\draw[line width=2pt] (13,1)--(14,1);
\draw[line width=2pt] (14,1)--(15,1);
\draw[line width=2pt] (15,1)--(15,2);
\draw[line width=2pt] (14,2)--(15,2);
\draw[dotted, line width=2pt] (15,2) to  (14,1);
\end{tikzpicture}
 \end{center}
\vspace{-1pc}
\caption{Forbidden subgraphs for uniquely $C_6$-saturated graphs.\label{c6forb2}}
\end{figure}

The method for $C_7$ is similar.

\begin{theorem}\label{thm:C7}
There are no nontrivial uniquely $C_{7}$-saturated graphs.
\end{theorem}

\begin{proof}
Let $G$ be a nontrivial uniquely $C_7$-saturated graph.
By Corollary~\ref{2conn} we may assume that $G$ is $2$-connected.
By Lemma~\ref{lem:C2t-2C2t-4}, $G$ does not contain $C_{12}$ or $C_{10}$.
By Lemma~\ref{lem:noevencyc}, $G$ contains $C_8$, $C_6$, or $C_4$.
Let $C$ be a longest cycle among the even cycles in $G$ with length at most $8$,
and let $R=V(G)-V(C)$.  In each of several cases, we obtain a contradiction.

{\bf Case 1:} {\it $C$ has length $8$.}
If $C$ has a chord joining vertices separated by distance $2$ or $3$ along $C$,
then $G$ contains $C_7$ or $H_{2,4}$ and is not uniquely $C_7$-saturated.
Hence any chord of $C$ joins opposite vertices on $C$.

By Lemma~\ref{cycpath}, $H_{4,2}\not\esub G$, so $R$ is an independent set.
Because $G$ is $2$-connected, each vertex in $R$ has at least two neighbors in
$V(C)$.  Consider $x\in R$.  If $x$ has neighbors on $C$ that are not
consecutive, then $G$ contains $H_{2,4}$, $C_7$, or $H_{3,3}$ and is not
uniquely $C_7$-saturated.  Hence every vertex of $R$ has exactly two neighbors
on $C$, and they are consecutive on $C$.

Since twins are forbidden, two vertices of $R$ cannot be adjacent to the same
consecutive pair.  However, two vertices of $R$ adjacent to distinct 
consecutive pairs yield $C_{10}$ in $G$, which is forbidden. 
We conclude $|R|\le 1$.  If $|R|=1$ and $C$ has a (diametric) chord, then
$H_{3,3}\esub G$, which is forbidden.  If $|R|=1$ and $C$ has no chord,
then adding any diametric chord completes no $7$-cycle.

Hence we may assume that $V(G)=V(C)$ and $C$ has only diametric chords.  Three
diametric chords of an $8$-cycle yield two $7$-cycles (each omits one of the
vertices not incident to a chord).  Hence $G$ has at most one chord $e$.
However, now no $7$-path connects two vertices not adjacent to either endpoint
of $e$.

{\bf Case 2:} {\it $C$ has length $6$.}
By Lemma~\ref{cycpath}, $H_{3,3}$ is not a subgraph of $G$.  Because $G$ is
$2$-connected, it follows that $G[R]$ has no component with at least three
vertices.  If $R$ is not independent, then there is a chordal path $P$ of
length $3$ connecting two vertices on $C$.  If those vertices are consecutive
or separated by distance $2$ along $C$, then $G$ contains $C_8$ or $C_7$, which
is forbidden.  If $P$ joins opposite vertices on $C$, then two $7$-paths join
the neighbors on $C$ of one of the endpoints of $P$.

Hence $R$ is independent.  Since $G$ is $2$-connected, each vertex of $R$ has
at least two neighbors in $V(C)$.  Consecutive neighbors on $C$ yield $C_7$.
Neighbors at distance $2$ along $C$ yield two $7$-paths with the same endpoints.
Hence every vertex of $R$ is adjacent precisely to two opposite vertices on
$C$.  Now any two vertices of $R$ are twins or yield $C_8$, both forbidden.

If $R=\nul$, then $G=K_6$, so we may let $R=\{x\}$.  The neighbors of $x$ are
opposite vertices $y$ and $z$ on $C$.  If $C$ has any non-diametric chord, then
two $7$-paths connect $x$ to some vertex on $C$.  A diametric chord other than
$yz$ creates $C_7$.  Hence the only possible chord is $yz$.  Now $\{y,z\}$ is a
separating set in $G$ such that $G-\{y,z\}$ has three components, and the
addition of a chord of $C$ incident to $y$ or $z$ cannot complete a spanning
cycle in $G$.

{\bf Case 3:} {\it $C$ has length $4$.}
Since $G$ is $2$-connected, there is a chordal path joining two vertices
of $C$.  If $V(C)$ is a clique, then a chordal path of length $3$, $4$, or at
least $5$ creates copies of $C_6$, $C_7$, or $H_{2,4}$, respectively, all
forbidden.  Hence every chordal path has length $2$.  Since $|V(G)|\ge 7$, we
conclude that $G$ contains $C_6$ or twins, both forbidden. 
We may therefore assume that $C$ is a $4$-cycle whose chords are not
both present.  Let $u$ and $v$ be nonconsecutive on $C$ such that
$uv\notin E(G)$, and let $x$ and $y$ be the other vertices of $C$.

The $7$-path $P$ with endpoints $u$ and $v$ also visits $x$ and $y$, since
otherwise $C_8\esub G$, which was forbidden in Case 1.  Let $V(C)$ occur in the
order $u,x,y,v$ along $C$.  The path $P$ uses exactly three vertices of $R$.
No matter how the three extra vertices are allocated to the three subpaths
connecting vertices of $C$, a $6$-cycle is created in $G$, excluded by Case 2.
\end{proof}

We have also proved there are no uniquely $C_8$-saturated graphs.  The proof
uses the approach above, but more cases and details are needed to exclude the
shorter even cycles.  Hence we omit the proof.

\section{Finitely Many Uniquely $C_t$-saturated graphs}\label{Sec:finite}

In this section, we present the proof of Theorem~\ref{thm:finite} that for
any $t\ge 6$ there are only finitely many uniquely $C_t$-saturated graphs.
The main idea is to reduce the problem to the $2$-connected case, showing
that if there are finitely many uniquely $C_t$-saturated graphs that are
$2$-connected, then there are finitely many uniquely $C_t$-saturated graphs.
For the first step, we restrict the ways that $2$-connected uniquely
$C_t$-saturated graphs can be combined.

\begin{lemma}\label{notwo}
If $t\ge6$ and $G$ is a $2$-connected uniquely $C_t$-saturated graph, then no
uniquely $C_t$-saturated graph $F$ has blocks $G'$ and $G''$ isomorphic to $G$
such that $G'$ and $G''$ share a cut-vertex of $F$ that can be viewed as the
same vertex of $G$ in $G'$ and $G''$.
\end{lemma}

\begin{proof}
Let $x$ be the vertex of $G$ in both $G'$ and $G''$, no edge of $F$ joins
$V(G'-x)$ and $V(G''-x)$.  For $y\in V(G-x)$, let $y'$ and $y''$ be the
corresponding vertices in $G'$ and $G''$.  If $G$ has distinct paths from $y$
to $x$ with lengths summing to $t-1$, then $F$ has two $t$-paths with endpoints
$y'$ and $y''$.  Hence the unique $t$-path with endpoints $y'$ and $y''$ 
consists of copies in $G'$ and $G''$ of a unique $(t+1)/2$-path $P_y$ in $G$
with endpoints $y$ and $x$.  This uniqueness implies that $G$ is not complete
(also $t$ is odd).  We consider two cases.

{\bf Case 1:} {\it $x$ is adjacent to all of $V(G-x)$.}
Let $\th=(t-1)/2$, so each $P_y$ is a $(\th+1)$-path.  Since $P_y$ is unique
and $x$ dominates $V(G-x)$, each $y\in V(G-x)$ starts exactly one $\th$-path in
$G-x$; it is $P_y-x$.  Let $z$ be the other endpoint of $P_y-x$.  Vertex $y$
cannot have a neighbor in $G-x$ outside $P_y$, since $G$ would then have
distinct paths from $z$ to $x$ with lengths $\th+1$ and $\th-1$ having sum
$t-1$.  Also $y$ cannot have a neighbor on $P_y$ other than its neighbor in
$P_y$, since distinct $\th$-paths in $G-x$ would then start at $z$.  Hence
$d_{G-x}(y)=1$.  With $y$ chosen arbitrarily, $G-x$ is $1$-regular.  Hence
$2=\th=(t-1)/2$, so this case requires $t=5$.


{\bf Case 2:} {\it $x$ has a nonneighbor $y$ in $G$.}
Since $G$ is not complete, by Lemma~\ref{lem:noevencyc} $G$ has an even cycle
$C$ of length at most $2t-6$.  Since $G$ is connected, there is a shortest
path $Q$ connecting $x$ to the copy of $C$ in $G'$.  Since $x$ has a nonneighbor
$y$ in $G$, there is a unique $t$-path $P$ in $G''$ with endpoints $x$ and
$y''$.  Letting $2k$ be the length of $C$, the subgraph $C\cup Q\cup P$ of $F$
contains $H_{k,t-k-1}$, which is forbidden by Lemma~\ref{cycpath}.
\end{proof}

\begin{lemma}\label{finite1}
If there are finitely many $2$-connected uniquely $C_t$-saturated graphs,
then there are finitely many uniquely $C_t$-saturated graphs.
\end{lemma}

\begin{proof}
The diameter of a $C_t$-saturated graph is at most $t-1$.
Hence the diameter of the block-cutpoint trees of $C_t$-saturated graphs
is also bounded; that is, a $C_t$-saturated graph cannot contain a path that
contains edges from more than $t-1$ blocks.

Every block of a uniquely $C_t$-saturated graph is uniquely $C_t$-saturated.
With finitely many $2$-connected uniquely $C_t$-saturated graphs, the number of
vertices in any single block of a uniquely $C_t$-saturated graph is bounded.
If there are infinitely many uniquely $C_t$-saturated graphs, then they must
exist with arbitrarily many blocks.  With bounded diameter in the
block-cutpoint tree, they must have block-cutpoint trees with arbitrarily many
leaves.

Since the distance between leaves is bounded, there must be arbitrarily many
leaf blocks having a common cutvertex.  Since the number of possible leaf
blocks is bounded, there must exist instances with arbitrarily many isomorphic
leaf blocks having a common cut-vertex.  Since the number of vertices in the
blocks are bounded, there must be instance with two isomorphic leaf blocks
sharing a cut-vertex that has the same identity in each of the two blocks.
This contradicts Lemma~\ref{notwo}.
\end{proof}

To complete the proof of the theorem, we need to show that there are
finitely many $2$-connected uniquely $C_t$-saturated graphs.  We do this
by bounding the number of vertices in such a graph.  Two lemmas are needed.

The first extends Lemma~\ref{twins}.  When $G$ has twins, it has an
automorphism exchanging the twins but leaving all other vertices fixed.
The twins are components of $G-S$, where $S$ is their common neighborhood.
We next consider a situation in which $G-S$ contains four isomorphic components.
Given a set $S\esub V(G)$ and $v\in V(G)-S$, let a {\it $v,S$-path} be a path
connecting $v$ to a vertex in $S$ with no internal vertices in $S\cup\{v\}$.

\begin{lemma}\label{lem:isomorphcompnts}
Given $t\ge8$, let $G$ be a $2$-connected graph with $S\subset V(G)$.  Suppose
that $G-S$ has distinct isomorphic components $F_1$, $F_2$, $F_3$ and $F_4$
such that for all $i\in \{2,3,4\}$ there is an automorphism $\varphi_i$ of $G$
such that (1) $\varphi_i^2$ is the identity, (2) $\varphi_i(F_1)= F_{i}$, and 
(3) $\varphi_i$ fixes all vertices outside $F_1\cup F_i$.  If $G$ is
uniquely $C_t$-saturated, then every vertex of $F_1$ that has a neighbor in $S$
starts some path in $F_1$ with length $t-2$.
\end{lemma}

\begin{proof}
Let $\th=(t-1)/2$.  We first prove that for every $x_1\in V(F_1)$ there is an
$x_1,S$-path of length $\th$.  If this fails for some $x_1\in V(F_1)$, then
let $x_2=\varphi_2(x_1)$.  Since $G$ is uniquely $C_t$-saturated, it contains a
unique $t$-path $P$ with endpoints $x_1$ and $x_2$.  For $i\in\{1,2\}$, let
$P^i$ be the $x_i,S$-path contained in $P$.  If $P^1$ and $P^2$ have the same
endpoint in $S$, then they have different lengths, since $G$ has no
$x_1,S$-path of length $\th$.  Otherwise, they have different endpoints in $S$.
In both cases, a second $t$-path with endpoints $x_1$ and $x_2$ consists of $P$
with each $P^i$ replaced by $\varphi_2(P^{3-i})$.  This contradiction proves
the claim.

Hence for all $x\in F_1$ there is an $x,S$-path of length $\th$.  Let $x$ be a
vertex of $F_1$ having a neighbor $y\in S$.  Suppose that $x$ starts no path in
$F_1$ with length $t-2$.  Since $G$ contains an $x,S$-path $P$ of length $\th$,
we may choose $z\in V(F_1)$ at distance $\th-1$ from $x$ along $P$.  Let $P'$
be the $\th$-path from $x$ to $z$ along $P$.  Because $G$ is $2$-connected,
$G-y$ has a shortest path connecting $V(P')$ and $S$; call it $Q'$, with
endpoints
$x'\in V(P')$ and $y'\in S$.  Let $Q$ be the path from $y$ to $y'$ consisting
of the edge $yx$, the subpath of $P'$ from $x$ to $x'$, and $Q'$ (see Figure~\ref{fig:F1}).  Let $k$ be
the length of $Q$.  Since $x$ starts no path of length $t-2$ in $F_1$,
we have $k\le t-1$.

\begin{figure}[htb]
 \begin{center}
 \begin{tikzpicture}
 
 \path
(0,0) node [line width=2pt,shape=ellipse,minimum width=6cm, minimum height=2cm,draw]{};
\fill (-4,2.5) circle (4pt);
\node at (-4.5,2.5) {$x$};
\fill (4,2.5) circle (4pt);
\node at (4.5,2.5) {$z$};
\draw[line width=1.5pt] (-4,2.5)--(4,2.5);
\fill (-2.5,.2) circle (4pt);
\node at (-2.2,0) {$y$};
\fill (-.5,.5) circle (4pt);
\node at (-.5,.1) {$y'$};
\node at (0,-.5) {$S$};
\draw[line width=2pt] (-4,2.5)--(-2.5,.2);
\draw[line width=2pt] (0,2.5)--(-.5,.5);
\fill (0,2.5) circle (4pt);
\node at (0,2.85) {$x'$};
\fill (3,2.5) circle (4pt);
\node at (3,2.1) {$w$};
\draw[line width=3pt] (-2.5,.2)--(-4,2.5)--(0,2.5)--(-.5,.5);
\node at (-2,1.5) {$Q$};
\draw[line width=1pt, ->] (-2.3,1.55)--(-3,1.55);
\draw[line width=1pt, ->] (-1.7,1.55)--(-.6,1.55);
\draw[line width=1pt, ->] (-2,1.8)--(-2,2.35);
 \draw[decorate,decoration={brace,amplitude=10pt,raise=1pt},yshift=0pt] (.2,2.4) -- (-.3,.6) node [midway,yshift=-5pt,xshift=18pt]{$Q'$};
  \draw[decorate,decoration={brace,amplitude=10pt,raise=1pt},yshift=0pt] (-4,2.9) -- (4,2.9) node [midway,yshift=17pt]{$P'$};
\end{tikzpicture}
 \end{center}
\vspace{-1pc}
\caption{Paths in $F_1$.
The bold path is $Q$, which overlaps with a portion of $P'$.\label{fig:F1}}
\end{figure}
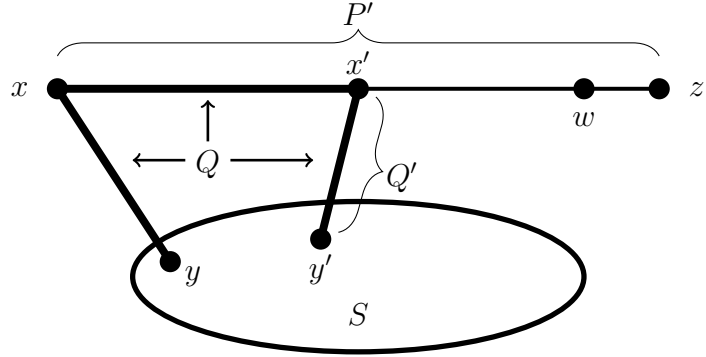

Note that $\varphi_2(Q)\cup\varphi_3(Q)$ is a cycle $C$ of length $2k$ in $G$,
and the union of $\langle y,x\rangle$ with $P'$ is a path of length $\th$.
If $k\ge \th$, then $C\cup\langle y,x\rangle\cup P'$ contains $H_{k,t-k-1}$.

Hence we may assume $k<\th$.  The portion of $Q$ along $P'$ has length at most
$k-2$.  Hence the path from $y'$ to $z$ in $Q'\cup P'$ has length at least
$\th-k+2$.  Let $\hat P$ be its subpath of length $\th-k$ starting from $y'$,
and let $w$ be the other endpoint of $\hat P$.  For $i\in\{2,3\}$, the
concatenation of $P'$, $xy$, $\varphi_i(Q)$, and $\varphi_4(\hat P)$ has
$\th+1+k+\th-k$ vertices; hence it is a $t$-path with endpoints $z$ and
$\varphi_4(w)$.  Again this contradicts $G$ being uniquely $C_t$-saturated, so
$x$ must start a path in $F_1$ with length $t-2$.
\end{proof}

\begin{lemma}\label{fin2conn}
There are finitely many $2$-connected uniquely $C_t$-saturated graphs.
\end{lemma}
\begin{proof}
By Theorems~\ref{thm:C6} and~\ref{thm:C7}, there are no uniquely
$C_6$-saturated or $C_7$-saturated graphs.  Hence we may assume $t\ge 8$.
It suffices to prove that the number of vertices in a $2$-connected uniquely
$C_t$-saturated graph $G$ is bounded.  In order to prove this, we prove that
the maximum degree in such a graph is bounded.  Since the diameter of
a uniquely $C_t$-saturated graph is less than $t$, this bounds the number of
vertices.

By Lemma~\ref{lem:noevencyc}, $G$ contains an even cycle $C$ of length at most
$2t-6$.  Let $C$ have length $2k$, and let $S=V(C)$.  By Lemma~\ref{cycpath},
$G$ does not contain $H_{k,t-k-1}$, and hence all paths leaving $S$ have length
at most $t-k-2$.

Let $R_{t-k-i}$ be the set of vertices $v$ outside $S$ such that longest
$v,S$-paths have length $i$.
Because $G$ is connected and $H_{k,t-k-1}\not\subseteq G$, it follows that every vertex of $G-S$ lies in $R_{t-k-i}$ for some $i$ with
$2\le i\le t-k-1$.
Also set $R_0=S$; this is 
$R_{t-k-i}$ for $i=t-k$.  We proceed by induction on $i$ to prove the
existence of $c_i$ such that $d_G(v)\le c_i$ when $v\in R_{t-k-i}$ and $2\le i\le t-k$.

For $v\in R_{t-k-2}$, the neighbors of $v$ lie in $S$ or on a $v,S$-path of
length $t-k-2$, since $H_{k,t-k-1}\notin G$.  Thus $d_G(v)\le t+k-3\le 2t-6$,
and we can set $c_2=2t-6$.

Now consider $v\in R_{t-k-i}$, where $3\le i\le t-k$.  Let $P$ be a $v,S$-path
of length $t-k-i$, and let $S'=S\cup V(P)$.  Let $N'(v)=N(v)-S'$.  By
Lemma~\ref{twins}, $G$ does not contain twins, so at most $2^{t+k-i-1}$
vertices in $N'(v)$ have neighborhoods contained in $S'$.  Let $N''(v)$ be the
set of vertices in $N'(v)$ having a neighbor outside $S'$, so
$|N''(v)|\ge d_G(v)-(t+k-i-1)-2^{t+k-i-1}$.

If any component of $G-S'$ contains an $(i-1)$-path starting at a vertex of
$N''(v)$, then $G$ contains $H_{k,t-k-1}$ and is not uniquely $C_t$-saturated.
Let $F$ be the subgraph of $G-S'$ that consists of the components of $G-S'$ that contain vertices in $N''(v)$.
Hence each vertex of $F$ lies in $R_{t-k-j}$ for some $j$ with $2\le j<i$.
By the induction hypothesis, $F$ has maximum degree bounded by
$\max_{j<i}c_j$.  With also bounded diameter, the number of vertices in a
component of $F$ is bounded by some value $h(t)$.

Let $F'$ be a possible component of $F$.  For each component of $F$
isomorphic to $F'$, we can list the neighborhood in it of each vertex of $S'$;
there are $(2^{|V(F')|})^{|S'|}$ possible such lists.  If $F$ has more than
$3\cdot 2^{|V(F')|\cdot|S'|}$ components isomorphic to $F'$, then some four of
them yield the same list.  Each vertex of $S'$ has the same neighborhood in
these four components, so together with $S'$ they satisfy the conditions of
Lemma~\ref{lem:isomorphcompnts}.  The resulting path with length $t-2$ from a
vertex of $N''(v)$ would contradict $G$ being uniquely $C_t$-saturated.

Since $|S'|=t+k-i<2t$, we have at most $3\cdot 2^{2t\cdot h(t)}$ components of $F$ isomorphic to $F'$, and there are fewer than $2^{\binom{h(t)}2}$ isomorphism
classes of graphs with at most $h(t)$ vertices.  Hence we have a bound (in
terms of $t$) on $|N''(v)|$ and hence also a bound $c_i$ on $d_G(v)$.

Since every vertex of $G$ lies in $R_{t-k-i}$ for some $i$ with
$2\le i\le t-k$, we have established $\max_{2\le i\le t-k} c_i$ as a bound
on the degrees of all vertices in $G$.
\end{proof}

Lemmas~\ref{finite1} and~\ref{fin2conn} complete the proof of
Theorem~\ref{thm:finite}: there are finitely many uniquely $C_t$-saturated
graphs.

\end{document}